\documentclass[dvips]{amsart}

\usepackage{graphicx}
\usepackage{amssymb}
\usepackage{amsmath}

\usepackage{tikz-cd}
\usepackage[a4paper, top=3.2cm, left=2cm, right=2cm, bottom=3cm]{geometry}

\usepackage[english]{babel}
\usepackage[utf8]{inputenc}
\usepackage[T1]{fontenc}
\usepackage{pdfsync}

\newtheorem{theorem}{Theorem}[section]

\newtheorem{definition}[theorem]{Definition}
\newtheorem{example}[theorem]{Example}

\newtheorem{proposition}[theorem]{Proposition}
\newtheorem{remark}[theorem]{Remark}

\numberwithin{equation}{section}

\def\e{{\varepsilon}}
\def\l{{\lambda}}
\def\i{{\iota}}

\def\Lieg{{\mathfrak{g}}}
\def\Lieh{{\mathfrak{h}}}
\def\gg{{\mathfrak{g}}}
\def\hh{{\mathfrak{h}}}
\def\Hom{{\mathrm{Hom}}}

\newcommand{\Lie}{\mathcal{L}}          

\tikzset{commutative diagrams/.cd,
mysymbol/.style={start anchor=center,end anchor=center,draw=none} }

\begin{document}

\title{Stability of Lie group Homomorphisms and Lie Subgroups}
 
\author{Cristian Camilo C\'ardenas} 
\author{Ivan Struchiner}
\address{Ivan Struchiner  \hfill\break\indent 
Universidade de S\~{a}o Paulo\\
Instituto de Matem\'{a}tica e Estat\'{\i}stica, \hfill\break\indent
 Rua do Mat\~{a}o 1010, 05508-090, S\~{a}o Paulo, SP, Brazil}


\email{ivanstru@ime.usp.br}

\address{Cristian Camilo C\'ardenas\hfill\break\indent 
Universidade Federal Fluminense\\
Instituto de Matem\'{a}tica e Estat\'{\i}stica, \hfill\break\indent
Rua Prof. Marcos Waldemar de Freitas Reis, S/n, 24210-201, Niter\'oi, RJ,  Brazil}
\email{ccardenascrist@gmail.com}

\thanks{The first author was partially supported by CAPES and CNPq during the development of this project at IME-USP. The second author was partially supported by FAPESP (2015/22059-2) and CNPq (307131/2016-5).}

%
%
%

\maketitle 

\begin{abstract}
We discuss a Moser type argument to show when a deformation of a Lie group homomorphism and of a Lie subgroup is trivial. For compact groups we obtain stability results\end{abstract}

\section{Introduction}

When studying an algebraic or geometric structure, a central problem is that of understanding how one such structure is related to the nearby ones. The main objective is to describe a neighbourhood of such structure in its moduli space. A first approximation to this problem is to study the space of structures which can be obtained from the original one through a small path in the moduli space. These paths give rise to families of structures which will be called deformations. This paper deals with deformations of Lie group homomorphism and Lie subgroups. More precisely, we are interested in understanding when a smooth family of Lie group homomorphisms or a smooth family of Lie subgroups represents a constant path in the corresponding moduli spaces. When this is the case, the deformation will be called a trivial deformation. Therefore, the first problem that we will deal with in this paper is that of determining when a deformation of a Lie group homomorphism (or a Lie subgroup) is trivial.

The notion of triviality of a deformation depends on the automorphism group that one considers. For Lie group homomorphisms $\phi: H \to G$ or Lie subgroups $H \subset G$, it is usual to consider the group of inner automorphisms of $G$ as the allowed group of automorphisms. Thus, for example, a smooth family of Lie group homomorphisms $\phi_\e: H \to G$ will be called trivial if there exists a smooth curve $\e \mapsto g_\e$ starting at the identity in $G$ such that $\phi_\e(h) = g_\e\phi_0(h)g_\e^{-1}$ for all $h \in H$. An analogous definition is made for deformations of a Lie subgroup $H \subset G$. We will relate the triviality of deformations with the (smooth) vanishing of certain classes in cohomology groups associated to $H$, $G$ and $\phi$.

\begin{theorem}\label{thm:homo:trivial}
Let $\phi_{\e}$ be a deformation of $\phi:H\rightarrow G$. Then for each $\lambda$ we obtain a 1-cocycle
\[X_\lambda(h) = d_{\phi_{\lambda}(h)}R_{\phi_{\lambda}(h)^{-1}}\frac{d}{d\e}|_{\e = \lambda}\phi_\e(h)\]
in the complex which computes the differentiable cohomology of $H$ with coefficients in the pullback by $\phi_\lambda$ of the adjoint representation of $G$.

Moreover, the deformation is trivial if and only if the family of 1-cocycles $X_\lambda$ can be smoothly transgressed (i.e., the cohomology classes vanish in a smooth manner).
\end{theorem}

To be precise, the result stated above holds under an extra completeness assumption. What we will actually prove is a local version of this result (see Theorem \ref{thm:homo:trivial}).
 
A similar result will be stated and proved for deformations of Lie subgroups (Theorem \ref{thm:trivial subgroup}). We also consider the problem of determining when a deformation $\phi_\e$ of $\phi: H \to G$ is trivial with respect to the full group of automorphisms of $G$. We will call such deformations \emph{weakly trivial}. In order to deal with this problem, we note that for each $\e$ we obtain a homomorphism
\[\phi_\e^*: H^*(G,\gg) \longrightarrow H^*_{\phi_\e}(H,\gg),\quad \phi_\e^*[c] = [\phi_\e^*(c)].\]
We will say that a family $[X_\e] \in H_{\phi_\e}^1(H, \gg)$ has a \emph{smooth pre-image in $H^1(G,\gg)$} if there exist  smooth families $Z_\e \in C^1_{\mathrm{cl}}(G,\gg)$ and $u_\e \in \gg$ such that
\[\phi_\e^*(Z_\e) = X_\e +\delta_{\phi_\e}(u_\e),\]
where $\delta_{\phi_\e}: \gg \to C^1_{\phi_\e}(H, \gg)$ denotes the differential of the complex computing the cohomology of $H$ with values in the pullback by $\phi_\e$ of the adjoint representation of $G$. In other words, the family of cohomology classes $[Z_\e] \in H^1(G,\gg)$ is a smooth pre-image of $[X_\e] \in H_{\phi_\e}^1(H, \gg)$. We will show that
\begin{theorem}
Let $\phi_\e: H \to G$ be a smooth family of Lie group homomorphisms and let $X_\e$ be its deformation cocycle. Then $\phi_\e$ is weakly trivial if and only if $[X_\e]$ has a smooth pre-image in $H^1(G,\gg)$.
\end{theorem}

We remark that the presence of the ``extra smoothness hypothesis" is unavoidable in our theorems. This is due to the geometric approach we use to prove the theorems (see the discussion on Moser's argument below). However, when $H$ is compact we can use a Haar measure on $H$ to provide explicit transgressions to $\delta_{\phi_\e}$. With this we obtain, for example, the following result.

\begin{theorem}\label{thm:compact:homo:stable}
Let $H$ be a compact Lie group. Then every Lie group homomorphism $\phi: H \to G$ admits only trivial deformations.
\end{theorem}

Similarly, we obtain the following result for compact Lie subgroups of $G$.

\begin{theorem}
Let $H \subset G$ be a compact Lie subgroup. Then every deformation of $H$ in $G$ is trivial.
\end{theorem}

\subsection*{Comparison to Existing Results} The problem of understanding the space of group homomorphisms and Lie subgroups near to a fixed one is not new. In \cite{NijenhuisRichardson}, the authors topologize the space of Lie group homomorphisms from $H$ to $G$ using the compact open topology and consider the $G$-action through inner automorphisms of $G$ on the space $\mathrm{Hom}(H,G)$ of Lie group homomorphisms. They sketch a proof of a theorem which states that if $H^1_{\phi}(H,\gg) = 0$ then the $G$-orbit of $\phi$ is open in  $\mathrm{Hom}(H,G)$, and therefore every nearby homomorphism is conjugate to $\phi$. A complete proof of this theorem can be found in \cite{lee1974deformationshom}, where the author allows $H$ to be a compactly generated locally compact group. This theorem is very closely related to our Theorems \ref{thm:homo:trivial} and \ref{thm:compact:homo:stable}. On the one hand, the conclusion of their theorem is stronger than ours since our conclusions only hold for paths of homomorphims while theirs is topological. If we could show that the space of Lie group homomorphisms is locally path connected in a neighbourhood of $\phi$ then one might expect to obtain their results from ours. However we do not know if such a result is true. On the other hand, our hypothesis is in a way weaker than theirs. To begin we do not ask for the vanishing of the entire cohomology group, but only the (smooth) vanishing of the deformation class to obtain our conclusion. Moreover, even if $\mathrm{Hom}(H,G)$ were locally path connected, our deformations would be more general than the deformations allowed in \cite{NijenhuisRichardson}. In fact, since they use the compact-open topology, their deformations are constrained to compact subsets of $H$, while we allow deformations which vary at infinity. However, when $H$ is compact it follows that $H^1_\phi(H,\gg)$ vanishes for all $\phi$ and we obtain a smooth transgressions of the deformation class of any deformation $\phi_\e$. In this way we obtain a parametric version of their result.

In \cite{Coppersmith}, the author deals with deformations of a Lie subgroup $H \subset G$. Similarly to our approach, the problem that is described is that of controlling the triviality of families of Lie subgroups. The key ingredient in the proof of his first main theorem is the Implicit Function Theorem which he applies under the condition that the whole cohomology $H^1(H, \gg/\mathfrak{h})$ vanishes. He then obtains that every deformation of $H$ is locally trivial. On the other hand, we again do not impose the vanishing of the entire cohomology group. Our result is valid for subgroups which may admit non-trivial deformations. We characterize the trivial deformations of such a subgroup as being those for which the deformation class vanishes smoothly in cohomology. When $H \subset G$ is a compact Lie subgroup we use an explicit transgression of the deformation class and recover the result of \cite{Coppersmith}. Coppersmith also discusses the obstructions to the existence of deformations with prescribed deformation class. We do not deal with this problem in this paper.

\subsection*{The Moser Deformation Argument} Our versions of the results on deformations of Lie group homomorphisms and Lie subgroups were made possible due to the technique we employ. Our approach is an adaptation of the Moser Deformation Argument to the context of deformations of homomorphisms and Lie subgroups. The Moser Deformation Argument is a classical technique in symplectic geometry used  to understand when a deformation of a symplectic structure is trivial.

In order to motivate the techniques used in this paper we briefly explain Moser's Deformation Argument in symplectic geometry. Recall that a symplectic structure on a manifold $M$ is a $2$-form $\omega \in \Omega^2(M)$ which is closed under the de Rham differential ($d\omega = 0$) and is non-degenerate in the sense that for every $x \in M$, if $u \in T_xM$ is a tangent vector then $\omega(u,v) = 0$ for all $v \in T_xM$ if and only if $u = 0$. Suppose that $\omega_\e$ is a smooth family of symplectic structures on $M$. Since each $\omega_\e$ is closed under the de Rham differential it follows that $\frac{d}{d\e}\omega_\e$ is also a cocycle in the de Rham cohomology of $M$ for all $\e$. The claim is then that there exists a smooth family $\phi_\e: M \to M$ (defined for small values of $\e$) of diffeomorphisms such that $\phi_0 = \mathrm{Id}_M$ and $\phi_\e^*\omega = \omega_\e$ if and only if the cocylcles  $\frac{d}{d\e}\omega_\e$ can be smoothly transgressed, i.e., if there exists a smooth family $\alpha_\e \in \Omega^1(M)$ of $1$-forms such that
\[d\alpha_\e = \frac{d}{d\e}\omega_\e\]
for all small values of $\e$.

To prove this claim one uses the following argument, which is known as the Moser Deformation Argument. Assume there exists $\phi_\e$ satisfying the conditions of the claim. Differentiating the condition
\[\phi_\e^*\omega = \omega_\e\]
with respect to $\e$ one obtains that
\[\Lie_{X_\e}\omega = \frac{d}{d\e}\omega_\e,\]
where $\Lie_{X_\e}$ denotes the Lie derivative with respect to the vector field $X_\e = \frac{d}{d\e}\phi_\e$. Using Cartan's magic formula and the fact that $d\omega = 0$ one obtains that
\[\Lie_{X_\e}\omega = d(\omega(X_\e)) = \frac{d}{d\e}\omega_\e\]
for all $\e$. But then one can take $\alpha_\e = \omega(X_\e)$ to conclude one of the implications of the claim. 

To prove the converse, suppose that there exists a smooth family $\alpha_\e$ of $1$-forms which satisfy $d\alpha_\e = \frac{d}{d\e}\omega_\e$. Since $\omega$ is non-degenerate, there exists a unique smooth time dependent vector field $X_\e$ on $M$ such that $\omega(X_\e) = \alpha_\e$ for all $\e$. Then the computations done in the first part of the proof show that the flow $\phi_\e$ of the time dependent vector field $X_\e$ satisfies $\phi_\e^*\omega = \omega_\e$ (for flows of time dependent vector fields see e.g., \cite{Time-varyingvectorfieldsandtheirflows}). To be precise, we remark that this argument is valid as long as the flow $\phi_\e$ is defined at all points for $\e \in [0,\e']$.

In this paper we use a similar approach to prove the triviality of deformations of Lie group homomorphisms and Lie subgroups. We identify the relevant cohomology theory controlling the deformation and prove that the deformation cocycle can be smoothly transgressed if and only if the deformation is (locally) trivial.

We remark that a similar approach has also been used in \cite{CMS} to understand when a deformation of Lie groupoids (and in particular Lie groups) is trivial. Also in \cite{Lory} these techniques were used to address deformations of Lie group representations.

This paper is organised as follows. In Section \ref{gp cohomology} we recall the definition of the differentiable cohomology of a Lie group with values in a representation and describe the special cases which are relevant for controlling the deformation problems. In Section \ref{results on morphisms} we describe our results for deformations of Lie group homomorphisms and in Section \ref{results on subgroups} we describe the results for Lie subgroups.

\subsection*{Acknowledgements} We would like to thank Jo\~ao Nuno Mestre for his valuable comments on a first version of this paper.


\section{Lie group cohomology}\label{gp cohomology}


In this section we briefly recall the definition of the differentiable cohomology of a Lie group $G$ with coefficients in a representation $\rho: G \to \mathrm{GL}(V)$, and describe the representations that will be useful for our purposes. 

Let $G$ be a Lie group and $\rho:G\rightarrow GL(V)$ be a representation of $G$. The \textbf{cochain complex of} $G$ \textbf{with coefficients in} $V$ is defined as follows: the cochains $C^{k}(G,V)$ of degree $k$ are smooth functions on $G^{k}$ (the cartesian product of $k$ copies of $G$) with values in $V$. The differential $\delta_{\rho}:C^{k}(G,V)\rightarrow C^{k+1}(G,V)$ is defined by
\begin{align*}
\delta_{\rho}c(g_1,...g_{k+1}):=\rho(g_1)c(g_2,...,g_{k+1})&+\sum^{k}_{i=1}(-1)^{i}c(g_{1},...g_ig_{i+1},...g_{k+1})\\
&+(-1)^{k+1}c(g_1,...,g_{k}).
\end{align*}
For any representation $\rho$ of $G$, we have $\delta_{\rho}^{2}=0$; the resulting cohomology is denoted by $H^{k}_{\rho}(G,V)$. The following examples will be useful for us.

\begin{example}
The adjoint complex of $G$ is the complex obtained by taking $V$ to be the Lie algebra $\Lieg$ of $G$, and $\rho$ to be the adjoint representation of $G$ on $\gg$. In this case, the differential of the cochain complex will be simply denoted by $\delta$, and the cohomology groups will be denoted by $H^{k}(G,\Lieg)$.
\end{example}

\begin{example}
If $\phi:H\rightarrow G$ is a homomorphism, then the adjoint representation of $G$ pulls-back to a representation of $H$, $\rho=\mathrm{Ad}\circ\phi:H\rightarrow GL(\Lieg)$. In this case, we will denote the differential of the resulting cochain complex by $\delta_{\phi}$ and its cohomology by $H^{*}_{\phi}(H,\Lieg)$.
\end{example}

\begin{example}
When $H\subset G$ is a Lie subgroup with Lie algebra $\hh \subset \gg$, the adjoint action of $H$ on $\gg$ induces a representation of $H$ on the quotient vector space $V = \Lieg/\Lieh$. In this case, we will  again denote the differential of the cochain complex by $\delta$ and we will denote the resulting cohomology groups by $H^{*}(H,\Lieg/\Lieh)$.
\end{example}

\begin{remark}
By putting together the two previous examples, if $\phi:H\rightarrow G$ is a homomorphism and $\Lieh$ is the Lie algebra of $H$, one obtains a representation of $H$ on $\gg/\tilde{\hh}$, where $\tilde{\hh}$ is the Lie algebra of $\phi(H)$. The differential of the resulting complex will be denoted by $\bar{\delta}_{\phi}$ and its cohomology groups will be denoted by $H^{k}_{\phi}(H,\Lieg/\tilde{\Lieh})$.
\end{remark}

\begin{remark}
Observe that, with the same notations as in the previous remark, there is a natural cochain-map (induced by $\phi$) between these complexes:
$$\phi_{*}: C^{*}(H,\Lieh)\rightarrow C^{*}_{\phi}(H,\tilde{\Lieh})\subset C^{*}_{\phi}(H,\Lieg).$$
\end{remark}

\section{Deformations and Stability of Lie Group Homomorphisms}\label{results on morphisms}
In this section we state and prove our main results for Lie group homomorphisms. Let $G$ and $H$ be Lie groups and let $\phi: H \to G$ be a Lie group homomorphism. 

\begin{definition}
A \textbf{deformation of $\phi$} is a smooth family $\phi_\e: H \to G$ of Lie group homomorphisms such that $\phi_0 = \phi$. Two deformations $\phi_\e$ and $\phi'_\e$ of $\phi$ are \textbf{equivalent} if there exists a smooth curve $g_\e: I \to G$ such that $g_\e(0) = 1$, and $\phi_\e = I_{g_\e} \circ \phi'_\e$ for all $\e \in I$, where $I_{g_\e}: G \to G$ denotes conjugation by $g_\e$. 
\end{definition}

\begin{remark}
We will also be interested in deformations which are only \textbf{locally equivalent}. For this, we only demand that $\phi_\e = I_{g_\e} \circ \phi'_\e$ for small enough values of $\e$.
\end{remark}

\begin{proposition}\label{def-homo}
Let $\phi_\e: H \to G$ be a deformation of $\phi$. Then for each $\lambda \in I$, the cochain
\[X_\lambda(h) = d_{\phi_{\lambda}(h)}R_{\phi_{\lambda}(h)^{-1}}\frac{d}{d\e}|_{\e = \lambda}\phi_\e(h) \in C^1_{\phi_\lambda}(H, \gg)\]
is a cocycle. 

Moreover, the cohomology class when $\lambda = 0$, $[X_0] \in H^1_{\phi}(H, \gg)$, only depends on the (local) equivalence class of the deformation $\phi_\e$.
\end{proposition}

\begin{proof}
We begin by proving that $X_\lambda$ is a cocycle. Let us denote by $m_H: H \times H \to H$ the multiplication on the Lie group $H$, and by $m_G: G \times G \to G$ the multiplication on $G$. Then, since $\phi_\e$ is a group homomorphism for all $\e$ we have that
\[\phi_\e(m_H(h_1, h_2)) = m_G(\phi_\e(h_1), \phi_\e(h_2)),\]
for all $\e \in I$. Differentiating both sides of this equation with respect to $\e$ at $\e = \lambda$ we obtain
\[\left.\frac{d}{d\e}\right|_{\e=\l}\phi_\e(m_H(h_1,h_2))-dm_G(\left.\frac{d}{d\e}\right|_{\e=\l}\phi_\e(h_1),\left.\frac{d}{d\e}\right|_{\e=\l}\phi_\e(h_2))=0.\]
Right translating back to the identity, i.e., applying $dR_{\phi_{\l}(m_H(h_1, h_2))^{-1}}$ we obtain the cocycle equation for $X_\l$.

Assume now that $\phi_\e$ and $\phi'_\e$ are equivalent deformations and let $g_\e$ be a curve in $G$ starting at the identity and such that
\[\phi_\e(h) = I_{g_\e} \circ \phi'_\e(h)\]
for all $h \in H$ and $\e \in I$. Equivalently,
\[m_G(\phi_\e(h), g_\e) = m_G(g_\e, \phi'_\e(h)).\]

Differentiating both sides of the equation at $\e = 0$ we obtain
\[\frac{d}{d\e}|_{\e=0}\phi_\e(h) + dL_{\phi(h)}u_0= \frac{d}{d\e}|_{\e=0}\phi'_\e(h) + dR_{\phi(h)}u_0,\]
where $u_0 = \frac{d}{d\e}|_{\e=0}g_\e \in \gg$. Right translating both sides of this equation by $\phi(h)^{-1}$, i.e., by applying $dR_{\phi(h)^{-1}}$ to both sides of the equation we obtain
\[X_0(h) + \mathrm{Ad}_{\phi(h)}u_0 = X'_0(h) + u_0,\]
therefore proving that $X_0 -X'_0$ is a coboundary concluding the proof.
\end{proof}

\begin{remark}\label{rmk: iso-lambda}
It would be pleasing to have a statement saying that $[X_\lambda] \in H^1_{\phi_\l}(H, \gg)$ only depends on the equivalence class of the deformation $\phi_\e$ for all $\l \in I$. The problem with such a statement is that even if $\phi_\e$ and $\phi'_\e$ are equivalent deformations of a Lie group homomorphism $\phi: H \to G$, the cohomology classes $[X_\lambda] \in H^1_{\phi_\l}(H, \gg)$ and $[X'_\lambda] \in H^1_{\phi'_\l}(H, \gg)$ live in different cohomology groups and we would have to identify both groups. One way around this difficulty is to note that $G$ acts on the space of Lie group homomorphisms $\Hom(H,G)$ by sending 
\[\psi\in\Hom(H,G) \longmapsto g \cdot \phi = I_g \circ \phi \in \Hom(H,G).\]
This action induces an isomorphism 
\[g_*: H_{\phi}^k(H,\gg) \longrightarrow  H_{g\cdot \phi}^k(H,\gg), \quad g_*[c] = [\mathrm{Ad}_g\circ c].\] 
Moreover, if $\phi_\e$ is a deformation of $\phi$, then $g\cdot\phi_\e$ is a deformation $g\cdot\phi$, and $g_*:H_{\phi}^1(H,\gg) \to  H_{g\cdot \phi}^1(H,\gg)$ maps the deformation class of $\phi_\e$ at time $0$ to the deformation class of $g\cdot \phi_\e$ at time $0$.

If $\phi_\e$ and $\phi'_\e$ are equivalent deformations of $\phi$, then by definition there exists a smooth curve $g_\e: I \to G$ such that $\phi_\e = g_\e\cdot \phi'_\e$ for all $\e \in I$. Thus, we may view $\phi_{\l+\e}$ and $g_\l\cdot \phi'_{\l+\e}$ as equivalent deformations of $\phi_\l$. It then follows that $(g_\lambda)_*[X'_\l] =[X_\l]$.
\end{remark}

In view of the proposition and remark above, we think of the cohomology classes $[X_\lambda] \in H^1_{\phi_\l}(H, \gg)$ as the \emph{velocity vector} of the equivalence class of the deformation $\phi_\e$ at time $\l$.

\begin{definition}
A deformation $\phi_\e: H \to G$ of $\phi$ is \textbf{(locally) trivial} if it is (locally) equivalent to the constant deformation $\phi'_\e \equiv \phi$ for all $\e \in I$.
\end{definition}

We wish to characterize the deformations of $\phi: H \to G$ which are trivial. Thinking of the cohomology class  $[X_\lambda]$ of a deformation $\phi_\e$ as its velocity vector, it is natural to suspect that deformations for which $[X_\lambda] = 0$ for all $\lambda$ are trivial. In order to take into acount also the smoothness of the deformation with respect to $\e$, we pose the following definition.

\begin{definition}
Let $\phi_\e: H \to G$ be a deformation of $\phi$ and for each $\e \in I$, let $c_\e \in C^1_{\phi_\e}(H, \gg)$ be cocycles. We say that the cohomology classes $[c_\e]$ \textbf{vanish smoothly} if there exists a smooth curve $u_\e \in \gg$ such that $\delta_{\phi_\e}(u_\e) = c_\e$ for all $\e \in I$. In this case we also say that $u_\e$ is a \textbf{smooth transgression} of $c_\e$.
\end{definition}

\begin{theorem}\label{thm:homo:trivial}
A deformation $\phi_\e: H \to G$ of a Lie group homomorphism $\phi: H \to G$ is locally trivial if and only if $[X_\e] \in H^1_{\phi_\e}(H,\gg)$ vanishes smoothly for small values of $\e \in I$.
\end{theorem}

\begin{proof}
Assume first that $\phi_\e$ is (locally) equivalent to the constant deformation, and let $g_\e$ be a curve in $G$ starting at the identity and such that $\phi_\e = I_{g_\e} \circ \phi$ for all $\e$. Differentiating both sides of this equation we obtain 
\begin{align*}
\left.\frac{d}{d\e}\right|_{\e=\l}\phi_\e(h)& =
\left.\frac{d}{d\e}\right|_{\e=\l}I_{g_\e}(\phi(h))\\
& = dR_{\phi(h)g_\l^{-1}}(\left.\frac{d}{d\e}\right|_{\e=\l}g_\e) +dL_{g_\l\phi(h)}\left[di\left(\left.\frac{d}{d\e}\right|_{\e=\l}g_\e\right)\right].\\
\end{align*}
Applying $dR_{\phi_\l(h)^{-1}} = dR_{g_\l\phi(h)^{-1}g_\l^{-1}}$ to both sides of the equation we obtain
\begin{align*}
X_\l(h)& = dR_{\phi_\l(h)^{-1}}\left(dR_{\phi(h)g_\l^{-1}}(\left.\frac{d}{d\e}\right|_{\e=\l}g_\e) +dL_{g_\l\phi(h)}\left[di\left(\left.\frac{d}{d\e}\right|_{\e=\l}g_\e\right)\right]\right)\\
& = dR_{g_\l^{-1}}(\left.\frac{d}{d\e}\right|_{\e=\l}g_\e) + \mathrm{Ad}_{\phi_\l(h)}\left(dL_{g_\l}\circ di(\left.\frac{d}{d\e}\right|_{\e=\l}g_\e)\right)\\
&= dR_{g_\l^{-1}}(\left.\frac{d}{d\e}\right|_{\e=\l}g_\e) + \mathrm{Ad}_{\phi_\l(h)}\left(di \circ dR_{g_\l^{-1}}(\left.\frac{d}{d\e}\right|_{\e=\l}g_\e)\right)\\
& = dR_{g_\l^{-1}}(\left.\frac{d}{d\e}\right|_{\e=\l}g_\e) - \mathrm{Ad}_{\phi_\l(h)}\left( dR_{g_\l^{-1}}(\left.\frac{d}{d\e}\right|_{\e=\l}g_\e)\right).
\end{align*}

Thus, by taking
\begin{equation}\label{Torigidity}
u_{\l}:= - \left.\frac{d}{d\e}\right|_{\e=\l}(g(\e)g(\l)^{-1}) = - dR_{g_\l^{-1}}(\left.\frac{d}{d\e}\right|_{\e=\l}g_\e) \in\ \Lieg,
\end{equation}
we obtain that
$$X_\l=\delta_{\phi_\l}(u_\l).$$
That is, the family of cocycles $X_\l$ is transgressed by the smooth family $u_\l$.

Conversely, assume that
\begin{equation}\label{*1}
X_{\l}=-\delta_{\phi_\l}(u_{\l}),
\end{equation}
where $u_\l$ is a smooth curve in $\Lieg=C^{0}_{\phi_\l}(H,\Lieg)$. Consider the time-dependent vector field $\overrightarrow{u}_{\e}$ on $G$, where each $\overrightarrow{u}_{\l}$ is the right invariant vector field which takes value $u_{\l}$ at the identity $1_G$ of $G$. Let $\epsilon'>0$ be such that the flow $\Phi^{\e,0}$ of $\overrightarrow{u}_{\e}$ is defined at $1_G\in G$ for all $\e \in (-\e',\e').$ Take $g_\e$ to be the integral curve from time $0$ to $\e$ of the time-dependent vector field $\overrightarrow{u}_\e$ starting at $1_G$, i.e., $g_\e:=\Phi^{\e,0}(1_G),\ (\left|\e\right|<\e').$ We will show now that $\phi_{\e}(h)=I_{g_{\e}}(\phi(h))$ for all small values of $\e$.

Consider the vector field on $G$ given by $Z_{\l}(g):=dR_{g}(u_{\l})-dL_{g}(u_{\l})$ and the vector field along $\phi_{\l}$, $\overrightarrow{X}_{\l}(h):=dR_{\phi_{\l}(h)}(X_{\l}(h)).$ Then, by $(\ref{*1})$, $\overrightarrow{X}_{\l}$ coincides with the pull-back of $Z_{\l}$ by $\phi_\l$. 

On the other hand, the deformation $I_{g_{\e}}\circ \phi$ of $\phi$ has associated the family of 1-cocycles
\begin{equation}\label{**1}
Y_{\e}:=-\delta_{I_{g_{\e}}\circ\phi}(u_{\e}).
\end{equation}

Equation $(\ref{**1})$ implies that the vector field $\overrightarrow{Y}_{\l} = dR_{I_{g_\l}\circ \phi(h)}(Y_{\l}(h))$ along $I_{g_\l}\circ \phi$ also coincides with the pull-back of $Z_{\l}$ by $I_{g_\l}\circ \phi$, for all $\l$. In other words, we have that $\e\mapsto\phi_{\e}(h)$ and $\e\mapsto I_{g_{\e}}\circ\phi(h)$ are integral curves of the time-dependent vector field $Z_{\e}$ passing through $\phi(h)\in G$ at time $\e=0$. Therefore, $\phi_{\e}(h)=I_{g_{\e}}(\i(h))$ for $\e$ small enough, as we claimed.
\end{proof}

The theorem above gives a characterization of the deformations of a homomorphism $\phi: H \to G$ which are (locally) trivial in terms of its deformation cocycle. We next show that if $H$ is compact, then any deformation of $\phi$ is trivial. We state this property by saying that $\phi$ is \textbf{stable}.

In order to prove our stability result, we will need to integrate a function with respect to a normalized left invariant Haar measure on $H$, i.e., a measure such that
\begin{itemize}
\item $\int_H f(h'h)dh = \int_Hf(h)dh$ for all $f \in \mathrm{C}^{\infty}(H)$ and $h' \in H$;
\item $\int_H  dh = 1$.
\end{itemize}
Any compact Lie group admits such a measure.
 
\begin{theorem}\label{RigidityMorphisms}
Let $H$ be a compact Lie group. Then any Lie group homomorphism  $\phi:H\to G$ is stable.
\end{theorem}

\begin{proof}Let $\phi_{\e}$ be any deformation of $\phi$. If $H$ is compact, each $H^{1}_{\phi_{\e}}(H,\Lieg)$ vanishes. A primitive of $X_{\e}$ is given by $u_{\e}=-\int_{H} X_{\e}(h)\:dh\in\Lieg$, where the integral is taken w.r.t. a normalized left invariant Haar's measure of $H$. In fact, since $X_\e$ is a 1-cocycle, $X_\e(h')=-\mathrm{Ad}_{\phi_\e(h')}X_\e(h)+X_\e(h'h)$. Thus by integrating one has 
\begin{align*}
X_\e(h')&=\int_{H}X_\e(h')dh\\
&=-\int_{H}\mathrm{Ad}_{\phi_{\e}(h')}X_\e(h)dh+\int_{H}X_\e(h'h)dh\\
&=-\mathrm{Ad}_{\phi_{\e}(h')}\int_{H}X_\e(h)dh+\int_{H}X_\e(h)dh\\
&= \mathrm{Ad}_{\phi_{\e}(h')}u_\e - u_\e \\
&=\delta_{\phi_\e}(u_\e)(h').
\end{align*}

It follows that $X_\e$ is smoothly transgressed. Moreover, since $H$ is compact the flow $\Phi^{\e,0}$ of the time dependent vector field obtained from the transgression of $X_\e$ is defined for all $\e \in I$. We can therefore apply Theorem \ref{thm:homo:trivial} to conclude that $\phi_\e$ is a trivial deformation of $\phi$.
\end{proof}


We can also apply our methods to study weak triviality of a deformation $\phi_\e$ of a Lie group homomorphism $\phi: H \to G$. 

\begin{definition}
A deformation $\phi_\e$ of a Lie group homomorphism $\phi: H \to G$ is said to be \textbf{weakly trivial} if there exists a smooth family $F_\e: G \to G$ of Lie group automorphisms such that $F_0 = Id_G$, and $\phi_\e = F_\e\circ \phi$ for all $\e \in I$. 
\end{definition}

Recall that a Lie group homomorphism $\phi: H \to G$ induces a pull-back map $\phi^*: H^k(G,\gg) \to H^k_{\phi}(H,\gg)$. The key to characterizing the weakly trivial deformations lies in understanding if the deformation cocycle of a deformation $\phi_\e$ lies in the the image of the pull-back map.

\begin{definition}
We will say that a family $[X_\e] \in H_{\phi_\e}^1(H, \gg)$ has a \textbf{smooth pre-image in $H^1(G,\gg)$} if there exist smooth families $Z_\e \in C^1_{\mathrm{cl}}(G,\gg)$ and $u_\e \in \gg$ such that
\[\phi_\e^*(Z_\e) = X_\e +\delta_{\phi_\e}(u_\e).\]
\end{definition}

\begin{theorem}\label{thm:loc-triv-homo}
Let $\phi_\e: H \to G$ be a smooth family of Lie group homomorphisms and let $X_\e$ be its deformation cocycle. Then $\phi_\e$ is locally weakly trivial if and only if $[X_\e]$ has a smooth pre-image in $H^1(G,\gg)$ for small values of $\e$.
\end{theorem}

\begin{proof}
Assume $\phi_\e=F_\e\circ\phi$, for a smooth family $F_\e$ of automorphisms of $G$, with $F_0=Id_G$. By applying $\frac{d}{d\e}$ to both sides of the equation we obtain

\begin{align*}
dR_{\phi_\e(h)}X_\e(h)&=\overrightarrow{Z}_\e(F_\e\circ\phi(h))\\
&=dR_{F_\e\circ\phi(h)}(Z_\e(F_\e\circ\phi(h)))\\
&=dR_{\phi_\e(h)}Z_\e(\phi_\e(h))
\end{align*}
where $\overrightarrow{Z}_\e =\frac{d}{d\e}F_\e$ is a vector field on $G$, and $Z_\e: G \to \gg$ is given by
\[Z_\e(g) = dR_{g^{-1}}\overrightarrow{Z}_\e(g).\]
It follows that $X_\e=\phi_\e^{*}Z_\e$, where $\phi_\e^{*}Z_\e$ is the 1-cocycle in $C^{1}_{\phi_\e}(H,\Lieg)$ obtained by pulling back the $1$-cocycle $Z_\e$ through the cochain map $\phi_\e^*$.

Conversely, assume that
\begin{equation}\label{preimage}
X_\e=\phi_\e^{*}Z_\e+\delta_{\phi_\e}(u_\e),
\end{equation}
for $Z_\e$ and $u_\e$ smooth families of elements in $C^{1}_{cl}(G,\Lieg)$ and $\Lieg$ respectively. Consider the time-dependent vector field $\overrightarrow{Z}_\e(g)=dR_{g}(Z_\e)$ on $G$. Define $F_\e=\Phi^{\e, 0}$ as the flow from time $0$ to $\e$ of $\overrightarrow{Z}_\e$, which exists for $\e$ small enough due to the right-invariance of each vector field $\overrightarrow{Z}_\e$. In fact, if $\tilde{\gamma}(r)=(\gamma(r),r)$, $r\in (-\e,\e)$, is the integral curve of the vector field $Z:=\overrightarrow{Z}_\e+\frac{\partial}{\partial\e}$ defined on $G\times I$ such that $\tilde{\gamma}(0)=(1_G,0)$, then, for any $g\in G$, $R_{g}(\tilde{\gamma}(r))=(R_{g}(\gamma(r)),r)$ is the integral curve of $Z$ starting in $(g,0)$; in other words $R_{g}(\gamma(r))$ is the integral curve of the time-dependent vector field $\overrightarrow{Z}_\e$ passing through $g\in G$ at time $\e=0$. Therefore, the flow $F_\e(g)$ is defined for every $g\in G$. 

We claim that $F_\e$ is a family of Lie group isomorphisms. In fact, if we denote by $m:G \times G \to G$ the multiplication on $G$, then the curves $F_\e(m(g_1, g_2))$ and $m(F_\e(g_1),F_\e(g_2))$ are integral curves of the same time-dependent vector field on $G$ starting at the same point at time $\e=0$ for all $g_1, g_2 \in G$. Indeed, on the one hand we have that
\begin{align*}
\left.\frac{d}{d\e}\right|_{\e=\l}m(F_\e(g_1),F_\e(g_2))&=dm\left(dR_{F_\l(g_1)}Z_\l(F_\l(g_1)),dR_{F_\l(g_2)}Z_\l(F_\l(g_2))\right)\\
&=dR_{F_\l(g_2)}dR_{F_\l(g_1)}Z_\l(F_\l(g_1))+dL_{F_\l(g_1)}dR_{F_\l(g_2)}Z_\l(F_\l(g_2))\\
&=dR_{m(F_\l(g_1),F_\l(g_2))}(Z_\l(m(F_\l(g_1),F_\l(g_2))))\\
&=\overrightarrow{Z}_\l(m(F_\l(g_1),F_\l(g_2))),
\end{align*}
where in the third equality we have used the fact that $Z_\l$ is a cocycle.

On the other hand, by definition we have that
$$\left.\frac{d}{d\e}\right|_{\e=\l}F_\e(m(g_1,g_2))=\overrightarrow{Z}_\l(F_\l(m(g_1,g_2))).$$
Thus, since $m(F_0(g_1),F_0(g_2))=F_0(m(g_1,g_2))$ the curves are the same. 

What we will show next is that $\phi'_\e = F_\e^{-1} \circ \phi_\e$ is a trivial deformation of $\phi$. That is, we will obtain a smooth curve $g_\e$ in $G$, starting at the identity, and such that $F_\e^{-1}\circ\phi_\e = I_{g_\e}\circ \phi$ for all $\e$. Therefore, we will have shown that $\phi_\e = F_\e\circ I_{g_\e}\circ \phi$ for all small values of $\e$ concluding the proof.

On the one hand, taking $\tilde{u}_\e\in\Lieg$ to be such that $u_\e=dF_\e(\tilde{u}_\e)$, equation \eqref{preimage} becomes
\begin{equation}\label{preimage2}
\begin{split}
X_\e&=\phi_\e^{*}(Z_\e)+\delta_{\phi_\e}(dF_\e(\tilde{u}_\e))\\
&=\phi_\e^{*}(Z_\e)+dF_\e(\delta_{F_{\e}^{-1}\phi_\e}(\tilde{u}_\e)).
\end{split}
\end{equation}

On the other hand, we set $X'_\e \in C^{1}_{\phi'_\e}(H,\Lieg)$ to be the family of deformation cocycles associated to the deformation $\phi'_\e = F_\e^{-1}\phi_\e$ of $\phi$. We claim that $X'_\e = \delta_{\phi'_\e}(\tilde{u}_\e)$, i.e., $X'_\e$ is smoothly transgressed. In fact,
\begin{equation*}
\left.\frac{d}{d\e}\right|_{\e=\l}\phi'_\e(h)=\left.\frac{d}{d\e}\right|_{\e=\l}F_\e^{-1}(\phi_\l(h))+dF_\l^{-1}(\left.\frac{d}{d\e}\right|_{\e=\l}\phi_\e(h)),
\end{equation*}
from where it follows that
\begin{equation*}
dR_{\phi'_{\l}(h)}X'_\l(h)=-dF_\l^{-1}(\overrightarrow{Z}_\l(\phi_\l(h)))+dF_\l^{-1}(dR_{\phi_\l(h)}X_\l(h)).
\end{equation*}
Thus, by applying $dR_{\phi'_{\l}(h)}^{-1}$ to both sides of the equation above we obtain
\begin{align*}
X'_\l(h)&=-dF_\l^{-1}(Z_\l(\phi_\l(h)))+dF_\l^{-1}(X_\l(h))\\
&=-dF_\l^{-1}(\phi_\l^{*}Z_\l)(h)+dF_\l^{-1}(X_\l)(h)\\
&=\delta_{\phi'_\l}(\tilde{u}_\l)(h),
\end{align*}
where the last equality follows from equation \eqref{preimage2}. It then follows from Theorem \ref{thm:homo:trivial} that $\phi'_\e$ is locally trivial concluding the proof of the theorem.
%
%

\end{proof}



\section{Deformations and Stability of Lie Subgroups}\label{results on subgroups}
In this section we state and prove our results on deformations and stability of Lie subgroups. Let $\iota: H \to G$ be an embedding of $H$ into $G$. Roughly speaking, a deformation of $H$ as a Lie subgroup of $G$ is a smooth family $(H_\e, \iota_\e)$ of embedded Lie subgroups such that $(H_0,\iota_0) = (H,\iota)$. In order to make this precise, we first explain what a deformation of a Lie group $H$ is (see also \cite{CMS} and references therein for the deformation theory of Lie groups and more generally of Lie groupoids).

\begin{definition}
Let $H$ be a Lie group. We denote its multiplication map by $m: H \times H \to H$ and its inversion map by $i: H \to H$. A \textbf{deformation of $H$} is a smooth family of maps $m_\e: H \times H \to H$, and $i_\e: H \to H$ such that $m_0 = m$, $i_0 = i$, and $H_\e = (H, m_\e, i_\e)$ is a Lie group for all $\e$.
\end{definition}

\begin{remark}
In principle one may also wish to allow the identity element to vary with $\e$. However, after composing with an isotopy of $H$ one would obtain an equivalent deformation where the identity element is fixed (see \cite{CMS}). For this reason we consider the identity element to be fixed for any deformation of $H$.
\end{remark}

\begin{remark}
In \cite{CMS}, the authors allow for more general deformations where $H_\e$ may vary smoothly (as a manifold). For this purpose, they consider $H_\e$ to be the fiber of a submersion $\widetilde{H} \to I$. Since we are interested in triviality of Lie subgroups, we will consider here only the case where $\widetilde{H} = H \times I$ (as manifolds). These deformations are called strict deformations of $H$ in \cite{CMS}.
\end{remark}

We can now proceed to define deformations of Lie subgroups.

\begin{definition}
Let $\iota: H \to G$ be an embedding of Lie groups. A \textbf{deformation of the Lie subgroup $(H, \iota)$} is a pair $(H_\e, \iota_\e)$ where $H_\e = (H, m_\e, \iota_\e)$ is a deformation of $H$, and $\iota_\e: H_\e \to G$ is a smooth family of embeddings of Lie groups. Two deformations $(H_\e, \iota_\e)$ and $(H'_\e, \iota'_\e)$ of $(H, \iota)$ are \textbf{(locally) equivalent} if there exists a smooth curve $g_\e$ starting at the identity element in $G$ such that $\iota_\e(H_\e) = I_{g_\e}( \iota'_\e(H'_\e))$ for all (small values of) $\e \in I$. 
\end{definition}

\begin{remark}\label{equiv-sub}
The equivalence of two deformations $(H_\e, \iota_\e)$ and $(H'_\e, \iota'_\e)$ of $(H, \iota)$ can be re-expressed as follows. There exists a family $F_\e: H'_\e \to H_\e$ of Lie group isomorphisms and a smooth curve $g_\e$ in $G$ starting at the identity element such that $F_0 = Id_H$ and
\[\iota_\e \circ F_\e = I_{g_\e} \circ \iota'_\e\]
for all $\e \in I$. This characterization of equivalent deformations of deformations of Lie subgroups in terms of homomorphisms will be useful in the proofs of the results presented below.
\end{remark}

\begin{proposition}\label{def-subgroup}
Let $(H_\e,\iota_\e)$ be a deformation of the Lie subgroup $\i:H\hookrightarrow G$. Then for each $\l$ the expression
$$\bar{X}_\l(h):=dR_{\iota_\l(h)^{-1}}\left.\frac{d}{d\e}\right|_{\e=\l}\iota_\e(h)\ \text{ mod } \hh_\l$$
defines a cocycle $\bar{X}_\l$ in the complex $C^{1}_{\iota_\l}(H_\l,\Lieg/\Lieh_\l)$, where $\hh_\l$ is the image under the differential of $\i_\l$ of the Lie algebra of $H_\l$. Moreover, $[\bar{X}_0] \in H^1_{\iota}(H,\Lieg/\Lieh)$ only depends on the (local) equivalence class of the deformation $(H_\e, \iota_\e)$.
\end{proposition}

\begin{proof}
The proof of the proposition is practically identical the proof of Proposition \ref{def-homo}. We give a sketch of the proof here and leave the details to the reader.

Denoting by $m_\e: H \times H \to H$ the multiplication of $H_\e$ and by $m_G: G \times G \to G$ the multiplication on $G$ we have that
\[\iota_\e(m_\e(h_1, h_2)) = m(\iota_\e(h_1), \iota_\e(h_2))\]
for all $h_1, h_2 \in H$, and for all $\e \in I$. Differentiating this equation with respect to $\e$ at $\e = \l$, right translating back to the identity and projecting the result onto $\gg/\hh_\l$ furnishes the cocycle equation for $\bar{X}_\l$. In this verification one must use the fact that
\[dR_{\iota_\l(h)^{-1}\iota_\l(g)^{-1}}\circ d\iota_\l\left(\left.\frac{d}{d\e}\right|_{\e = \l}m_\e(h_1,h_2)\right) \in \hh_\l.\]

For the proof of the second part of the proposition, assume that $(H_\e, \iota_\e)$ and $(H'_\e, \iota'_\e)$ are equivalent deformations of $(H, \iota)$. According to Remark \ref{equiv-sub} there exists a smooth family $F_\e: H'_\e \to H_\e$ of Lie group isomorphisms and a smooth curve $g_\e$ in $G$ starting at the identity element such that $F_0 = Id_H$ and
\[\iota_\e \circ F_\e = I_{g_\e} \circ \iota'_\e\]
for all $\e \in I$. Let $Y_0(h) = \left.\frac{d}{d\e}\right|_{\e = 0}F_\e(h)$, and $u_0 = \left.\frac{d}{d\e}\right|_{\e = 0}g_\e$. Differentiating the equation above w.r.t. $\e$ at $\e = 0$ we obtain
\[d\iota(Y_0) + \left.\frac{d}{d\e}\right|_{\e = 0}\iota_\e(h) + dL_{\iota(h)}(u_0) = dR_{\iota(h)}(u_0) + \left.\frac{d}{d\e}\right|_{\e = 0}\iota'_\e(h).\]
Right translating this expression back to the identity via $dR_{\iota(h)^{-1}}$ and projecting to $\gg/\hh$ we obtain
\[\bar{X}'_0(h) - \bar{X}_0(h) = \left(\mathrm{Ad}_{\iota(h)}(u_0) - u_0\right) \text{ mod }\hh \]
where we have used the fact that $dR_{\iota(h)^{-1}}\circ d\iota(Y_0(h))$ belongs to $\hh$. It follows that $$\bar{X}'_0 - \bar{X}_0 = \bar{\delta}_{\iota}(\bar{u}_0),$$ 
which concludes the proof.
\end{proof}

As in the previous section, we will consider the problem of characterizing locally trivial deformations in terms infinitesimal data. 

\begin{definition} 
A deformation $(H_\e, \iota_\e)$ of $(H, \iota)$ is \textbf{(locally) trivial} if it is (locally) equivalent to the constant deformation $(H'_\e, \iota'_\e) \equiv (H,\iota)$.
\end{definition}

\begin{definition}
Let $(H_\e, \i_\e)$ be a deformation of $(H,\i)$ and let $\bar{X}_\e \in C^1_{\i_\e}(H, \gg/\hh_\e)$ be its family of deformation cocycles. We will say that $\bar{X}_\e$ is \textbf{smoothly transgressed} if there exists a smooth curve $u_\e \in \gg$ such that $\bar{X}_\e = \bar{\delta}_{\i_\e}(\bar{u}_\e)$, where $\bar{u}_\e = u_\e \text{ mod }\hh_\e$.
\end{definition}
%

\begin{theorem}\label{thm:trivial subgroup}
Let $(H_\e, \iota_\e)$ be a deformation of an embedded Lie subgroup $\iota: H \hookrightarrow G$. Then $(H_\e, \iota_\e)$ is locally trivial if and only if $\bar{X}_\e$ can be smoothly transgressed for all small values of $\e \in I$.
\end{theorem}

\begin{proof}
The triviality of $(H_{\e}, \iota_\e)$ amounts to saying that
$$\i_\e\circ F_\e=I_{g_\e}\circ\i,$$ for a smooth family of isomorphisms $F_\e:H\rightarrow H_\e$.
Denote by $\overrightarrow{Y}_\e$ the time-dependent vector field on the manifold $H$ induced by the diffeomorphisms $F_\e$ and let $u_\l = \left.\frac{d}{d\e}\right|_{\e = \l}g_\e g_\l^{-1}$. Differentiating the equation above with respect to $\e$ at $\e = \l$ and right translating back to the identity we obtain
\[dR_{\i_\l(F_\l(h))^{-1}}\left(d\i_\l(\overrightarrow{Y}_\l(F_\l(h))) + \left.\frac{d}{d\e}\right|_{\e = \l}\i_\e(F_\l(h)) \right) = u_\l - \mathrm{Ad}_{\i_\l(F_\l(h))}(u_\l).\]
Projecting to $\gg/\hh_\l$ and using the fact that $dR_{\i_\l(F_\l(h))^{-1}}\left(d\i_\l(\overrightarrow{Y}_\l(F_\l(h)))\right) \in \hh_\l$ we obtain
\[\bar{X}_\l (F_\l(h)) = - \bar{\delta}_{\i_l}(\bar{u}_\l)(F_\l(h))\]
for all $h \in H$, and $\l \in I$. Since $F_\l$ is surjective, we conclude that
\[\bar{X}_\l =  \bar{\delta}_{\i_\l}(-\bar{u}_\l)\]
which proves that $\bar{X}_\e$ can be smoothly transgressed.


We will now prove the converse statement. Assume that $\bar{X}_\e$ can be smoothly transgressed. By definition, there exists a smooth curve $u_\e$ in $\gg$ such that
\[\bar{X}_\e(h) = \bar{\delta}_{\i_\e}(\bar{u}_\e)(h) = \left(\mathrm{Ad}_{\i_\e(h)}(u_\e) - u_\e\right)\text{ mod }\hh_\e\]
for all $h \in H$ and $\e \in I$.

Let $X_\e(h) = dR_{\i_\e(h)^{-1}}\left(\frac{d}{d\e}\i_\e(h)\right) \in \gg$. Since $X_\e(h)$ and  $\mathrm{Ad}_{\i_\e(h)}(u_\e) - u_\e$ project to the same element in $\gg/\hh_\e$, it follows that
\begin{equation}\label{deform.cocycle}
X_\e(h) - \mathrm{Ad}_{\i_\e(h)}(u_\e) + u_\e = -d_{\i_\e}(Y_\e(h)),
\end{equation}
where $Y_\e: H \to \hh$ is a smooth map.

Let $\overrightarrow{Y}_\e$ be the time dependent vector field on $H$ given by 
\[\overrightarrow{Y}_\e(h) = dR^\e_h(Y_\e(h)),\] 
where $R^\e_h(h') = m_\e(h', h)$ denotes right translation by $h$ in $H_\e$, and let $F_\e$ denote its flow from time $0$ to $\e$, which exists for $\e$ small enough due to the right-invariance of each vector field $\overrightarrow{Y}_\e$ (for the existence of the flow, see the analogous statement in the proof of Theorem \ref{thm:loc-triv-homo}).

We claim that for each $\e$, $F_\e: H \to H_\e$ is a Lie group isomorphism. In order to show this we must verify that
\[F_\e(h_1\cdot h_2) = m_\e(F_\e(h_1), F_\e(h_2))\]
for all $h_1, h_2 \in H$, and all $\e$, where $H_\e = (H, m_\e, i_\e)$.
It is clear that the equation holds when $\e = 0$, so we will prove that both sides are integral curves of the same time-dependent vector field defined on $H$. 

On the one hand, we have that
\begin{align*}
\left.\frac{d}{d\e}\right|_{\e=\l}m_\e(F_\e(h_1),F_\e(h_2))&=\left.\frac{d}{d\e}\right|_{\e=\l}m_\e(F_\l(h_1),F_\l(h_2))+ \\
&\ \ \ \ \ \ \ \ \ \ + dm_\l\left(dR^{\l}_{F_\l(h_1)}Y_\l(F_\l(h_1)),dR^{\l}_{F_\l(h_2)}Y_\l(F_\l(h_2))\right)\\
&=\left.\frac{d}{d\e}\right|_{\e=\l}m_\e(F_\l(h_1),F_\l(h_2))+dR^{\l}_{F_\l(h_2)}dR^{\l}_{F_\l(h_1)}Y_\l(F_\l(h_1))+\\
&\ \ \ \ \ \ \ \ \ \ +dL^{\l}_{F_\l(h_1)}dR^{\l}_{F_\l(h_2)}Y_\l(F_\l(h_2))\\
&=dR^{\l}_{m_\l(F_\l(h_1),F_\l(h_2))}\left(dR^{\l}_{m_\l(F_\l(h_1),F_\l(h_2))^{-1}}\left.\frac{d}{d\e}\right|_{\e=\l}m_\e(F_\l(h_1),F_\l(h_2))\right)+ \\
&\ \ \  +dR^{\l}_{m_\l(F_\l(h_1),F_\l(h_2))}\big(Y_\l(m_\l(F_\l(h_1),F_\l(h_2)))+\delta_\l(Y_\l)(F_\l(h_1),F_\l(h_2))\big)\\
&=dR^{\l}_{m_\l(F_\l(h_1),F_\l(h_2))}(Y_\l(m_\l(F_\l(h_1),F_\l(h_2))))\\
&=\overrightarrow{Y}_\l(m_\l(F_\l(h_1),F_\l(h_2))),
\end{align*}
where in the fourth equality we have used the fact that
$$\delta_\l(Y_\l)(h_1,h_2)=-dR^{\l}_{m_\l(F_\l(h_1),F_\l(h_2))^{-1}}\left(\left.\frac{d}{d\e}\right|_{\e=\l}m_\e(F_\l(h_1),F_\l(h_2))\right)$$
which follows from applying $\delta_{\i_\e}$ to  Equation \eqref{deform.cocycle} and using that
$$\delta_{\i_\e}X_\e(h_1,h_2)=d\i_\e\circ dR^{\e}_{m_\e(h_1,h_2)^{-1}}(\frac{d}{d\e}m_\e(h_1,h_2))$$
(see Proposition \ref{def-subgroup}).

On the other hand, by definition we have that
$$\left.\frac{d}{d\e}\right|_{\e=\l}F_\e(m(h_1,h_2))=\overrightarrow{Y}_\l(F_\l(m(h_1,h_2))).$$
Therefore, $m_\e(F_\e(h_1),F_\e(h_2))$ and $F_\e(m(h_1,h_2))$ are integral curves of the time-dependent vector field $\overrightarrow{Y}_\l$ on $H$ and they start at the same point at time $\e=0$, hence they are equal.

Define $\i'_\e:=\i_\e\circ F_\e$. We claim that the deformation cocycles of the deformation $\i'_\e$ of $\i$ can be smoothly transgressed. Indeed, it is straightforward to check that its associated 1-cocycles $X'_{\e}$ are
\begin{align*}
X'_{\e}(h)&=X_{\e}(F_{\e}(h))+d\i_{\e}Y_{\e}(F_\e(h))\\
&=\delta_{\i_{\e}}(u_{\e})(F_{\e}(h))\\
&=\delta_{\i_{\e}\circ F_{\e}}(u_{\e})(h)\\
&=\delta_{\i'_{\e}}(u_{\e})(h).
\end{align*}
Therefore, by Theorem \ref{thm:homo:trivial}, the deformation $\i'_{\e}$ is locally trivial. In other words, we have that $\i_{\e}\circ F_{\e}=I_{g_{\e}}\circ\i$ for small values of $\e$, from where it follows that $(H_{\e}, \i_\e)$ is locally trivial.
\end{proof}

\begin{remark}
We remark that the proofs presented above continue valid word by word if instead of considering embedded subgroups one considers immersed subgroups (and deformations by immersed subgroups).  
\end{remark}


The second main result about stability of this paper states that any compact Lie subgroup of any Lie group $G$ is stable as a Lie subgroup. In order to prove this result, we will need to use a \emph{normalized left invariant Haar system} on a deformation $H_\e$ of $H$. Such a system is a family of normalized left invariant Haar measures on $H_\e$ which depend smoothly on $\e$ in the sense that if $f: H \times I \to \mathbb{R}$ is smooth, then
\[\e \in I \mapsto \int_{H_\e}f_\e(h)dh\]
is a smooth function on $I$. Any deformation of a compact Lie group admits such a Haar System (see for example \cite{Renault} or \cite{Crainic}).

\begin{theorem}\label{rigiditysubgroups}
Let $G$ be a Lie group and let $\i:  H \hookrightarrow G$ be a compact Lie subgroup. Then any deformation $(H_\e, \i_\e)$ of $(H, \i)$ is trivial.
\end{theorem}

\begin{proof}
Let $(H_\e, \i_\e)$ be a deformation of $(H, \i)$. Consider the smooth family of functions $X_\e: H \to \gg$ given by
$$X_{\e}(h)=dR_{\i_{\e}(h)}^{-1}\left(\left.\frac{d}{d\l}\right|_{\l=\e}\i_{\l}(h)\right),$$
so that $X_\e \text{ mod } \hh_\e = \bar{X}_\e \in C_{\i_\e}^1(H,\gg/\hh_\e)$ is the family of deformation cocycles associated to the deformation $(H_\e, \i_\e)$. 

We define the $u_\e \in \gg$ by
\[u_\e = -\int_{H_\e}X_\e(h)dh.\]
Notice that since the Haar system is smooth, it follows that $u_\e$ is a smooth curve in $\gg$. By a computation identical to the one presented in the proof of Theorem \ref{RigidityMorphisms} we obtain that 
\[\bar{\delta}_{\i_\e}(\bar{u}_\e) = \bar{X}_\e,\]
where $\bar{u}_\e = u_\e \text{ mod }\hh_\e$.

It follows that $\bar{X}_\e$ can be smoothly transgressed and we can apply Theorem \ref{thm:trivial subgroup} to conclude the proof. The \emph{``global''} triviality of the deformation follows from the fact that $H$ is compact and therefore the flows of the vector fields used in the proof of Theorem \ref{thm:trivial subgroup} are defined for all $\e \in I$.
\end{proof}

\end{document}